\documentclass[11pt]{amsart}
\usepackage[letterpaper,margin=1in]{geometry}
\usepackage[T1]{fontenc}
\usepackage{lmodern}
\usepackage{amsmath,amssymb,amsthm,mathtools}
\usepackage{microtype}
\usepackage{needspace}
\usepackage{tikz}
\usepackage[colorlinks=true,linkcolor=blue!45!black,citecolor=blue!45!black,urlcolor=blue!45!black]{hyperref}
\numberwithin{equation}{section}
\newtheorem{theorem}{Theorem}[section]
\newtheorem{proposition}[theorem]{Proposition}
\newtheorem{lemma}[theorem]{Lemma}

\theoremstyle{remark}
\newtheorem{remark}[theorem]{Remark}
\newcommand{\Fq}{\mathbb F_q}
\newcommand{\Fqq}{\mathbb F_{q^2}}
\newcommand{\Aq}{\mathcal A_q}
\newcommand{\Aqq}{\mathcal A_{q^2}}
\newcommand{\Hq}{\mathcal H_q}
\newcommand{\Hqq}{\mathcal H_{q^2}}
\newcommand{\eps}{\varepsilon}
\DeclareMathOperator*{\Res}{Res}

\newcommand{\nK}{\mathrm{nK}}
\newcommand{\sumfamily}{\sum_{\substack{\chi\ \mathrm{primitive\ cubic}\\ \operatorname{genus}(\chi)=g}}}
\allowdisplaybreaks[1]
\title[A secondary term in cubic first moments]{A secondary term in the first moment of cubic $L$-functions over function fields}
\author{Ziwei Hong}
\address{School of Mathematics, Renmin University of China, Beijing, P.R. China}
\address{School of Mathematics and Statistics, UNSW Sydney, Sydney, Australia}
\email{hongziwei@live.com}
\date{}
\hypersetup{pdftitle={A secondary term in the first moment of cubic L-functions over function fields},pdfauthor={Ziwei Hong},pdfsubject={Cubic twisted first moments and explicit secondary terms},pdfkeywords={cubic L-functions, Gauss sums, secondary terms, twisted moments}}
\begin{document}
\begin{abstract}
Let $q$ be a fixed odd prime power with $q\equiv2\pmod3$. We obtain an explicit secondary term of order $q^{5g/6}$ in the first moment of primitive cubic $L$-functions over $\Fq(T)$, with error $O_{q,\eps}(q^{(23/28+\eps)g})$. Its real coefficient is periodic in $g$ modulo $6$ and is not identically zero. The proof continues the double Dirichlet series of Hong and Fang beyond the secondary poles by summing over the variable introduced by M\"obius inversion before taking absolute values. The residues are evaluated using the formulas of David, Florea and Lal\'in. The same method gives a uniform twisted first moment with error $O_{q,\eps}(q^{(23/28+\eps)g}|h|^{5/14+\eps})$.
\end{abstract}
\maketitle

\section{Introduction}

Throughout the paper, $q$ is a fixed odd prime power satisfying $q\equiv2\pmod3$. We consider
\begin{equation}\label{eq:moment}
\sumfamily L_q\!\left(\frac12,\chi\right),
\end{equation}
where the genus follows the convention in \cite{DFL,HF}. In the present non-Kummer family, the conductor has degree $g+2$ and $g$ is even. Our aim is to separate the secondary contribution of size $q^{5g/6}$ from the error term.

\subsection{Earlier work}
Over the Eisenstein field, Luo \cite{Luo} obtained a first moment asymptotic for a family of cubic Hecke characters. For primitive cubic Dirichlet characters over $\mathbb Q$, Baier and Young \cite{BY} proved a smoothed first moment formula with error $O(X^{37/38+\eps})$. Gao and Zhao \cite{GZ} subsequently obtained the exponent $7/8$ by double Dirichlet series under the generalized Lindel\"of hypothesis.

More recently, Hamdar \cite{Hamdar} separated a secondary term of order $X^{5/6}$, with error $O(X^{4/5+\eps})$ at the central point, for the thin family of cubic Hecke characters attached to squarefree Eisenstein integers congruent to $1$ modulo $9$.

David, Florea and Lal\'in \cite{DFL} proved the first moment asymptotic
\begin{equation}\label{eq:dfl-first}
\sumfamily L_q\!\left(\frac12,\chi\right)
=q^{g+2}\frac{\zeta_q(3/2)}{\zeta_q(3)}
\mathcal A_{\nK}(q^{-2},q^{-3/2})
+O_{q,\eps}\!\left(q^{(7/8+\eps)g}\right),
\end{equation}
where $\zeta_q(s)=(1-q^{1-s})^{-1}$ and $\mathcal A_{\nK}$ is an explicit Euler product. They also observed that a contribution of size $q^{5g/6}$ could be obtained by evaluating additional residues, but left it within the larger error term.

Hong and Fang \cite{HF} introduced a double Dirichlet series $A_3$ for the same family and recovered \eqref{eq:dfl-first} by analytic continuation and coefficient extraction. We follow their notation
\begin{equation}\label{eq:variables}
u=q^{-2s},\qquad v=q^{-w},\qquad A_3(u,v)=A_3(s,w).
\end{equation}
Their continuation reaches $\Re s>7/8$, short of the known secondary poles at $\Re s=5/6$. We enlarge this region to cross those poles and identify their residues in $A_3$.

\subsection{The main result}
We use the constants of David, Florea and Lal\'in \cite{DFL}. All products below run over monic irreducible polynomials $P\in\Fq[T]$, and $|P|=q^{\deg P}$. The factor in their first main term is
\begin{equation}\label{eq:AnK}
\begin{split}
\mathcal A_{\nK}(q^{-2},q^{-3/2})
={}&\prod_{\substack{P\\\deg P\text{ odd}}}\frac1{1+|P|^{-2}}\\
&\times\prod_{\substack{P\\\deg P\text{ even}}}
\frac{1+2|P|^{-1}-2|P|^{-5/2}}{(1+|P|^{-1})^2}.
\end{split}
\end{equation}
To specify the secondary term, we use the function $\mathcal K_{\nK}$
defined in \cite[(67)]{DFL}. We record its Euler product from
\cite[(66)--(67)]{DFL} in the present notation:
\begin{equation}\label{eq:KnK}
\begin{split}
\mathcal K_{\nK}(z)
={}&\prod_{\substack{P\\d=\deg P\text{ odd}}}
\frac{1-z^{2d}|P|^{-5/3}}
{(1+|P|^{-2})(1-z^{3d}|P|^{-3/2})}\\
&\times\prod_{\substack{P\\d=\deg P\text{ even}}}
\frac{(1-z^d|P|^{-5/6})
(1+2|P|^{-1}+z^d|P|^{-5/6}-2z^{3d}|P|^{-5/2})}
{(1+|P|^{-1})^2(1-z^{3d}|P|^{-3/2})}.
\end{split}
\end{equation}
As noted in \cite[p.~46]{DFL}, this product converges absolutely for
$|z|<q^{1/6}$, and the convergence is locally uniform there.

\begin{theorem}\label{thm:main}
Fix an odd prime power $q\equiv2\pmod3$, and let $0<\eps<1/84$. As $g$ tends to infinity through the even integers,
\begin{equation}\label{eq:main}
\begin{split}
\sumfamily L_q\!\left(\frac12,\chi\right)
={}&q^{g+2}\frac{\zeta_q(3/2)}{\zeta_q(3)}
\mathcal A_{\nK}(q^{-2},q^{-3/2})\\
&+q^{5g/6+5/3}\sum_{j=0}^2c_j\omega^{-j(g/2+1)}
+O_{q,\eps}\!\left(q^{(23/28+\eps)g}\right),
\end{split}
\end{equation}
where $\omega=e^{2\pi i/3}$ and
\begin{equation}\label{eq:cj}
c_j=\frac{(1+q^{1/3}\omega^j)\mathcal K_{\nK}(\omega^{-j})}
{3\zeta_q(3)(1-q^{1/6}\omega^{-j})},\qquad 0\le j\le2.
\end{equation}
The coefficients satisfy $c_0<0$ and $c_2=\overline{c_1}$. In particular, the real periodic coefficient
$\sum_{j=0}^2c_j\omega^{-j(g/2+1)}$ is not identically zero on the even integers.
\end{theorem}

The periodic coefficient has average $c_0<0$ over the three admissible
genus classes, so at least one class has a nonzero secondary term.
The exponent gap is $5/6-23/28=1/84$. All estimates are for fixed $q$.

\subsection{Twisted first moments}
Let $\mathcal C_3(g)$ denote the family of primitive cubic characters
in \eqref{eq:moment}, and put $m=g/2+1$. For a monic $h\in\Fq[T]$,
write $h=h_1h_2^2h_3^3$, where $h_1,h_2$ are squarefree and coprime.
The same continuation argument gives a uniform twisted first moment,
including its secondary residues.

\begin{theorem}\label{thm:twisted}
With $C_q(h)$ and $c_j(h)$ defined in \eqref{tw:C} and \eqref{tw:local-secondary}, respectively, for every sufficiently small $\eps>0$, uniformly in monic $h\in\Aq$
and integers $m=g/2+1\ge1$,
\begin{equation}\label{tw:theorem}
 \begin{split}
 \sum_{\chi\in\mathcal C_3(g)}
 \chi(h)L_q(1/2,\chi)
 ={}&q^{2m}C_q(h)
   +q^{5m/3}\sum_{j=0}^2c_j(h)\omega^{-jm}\\
 &+O_{q,\eps}\!\left(q^{(23/14+\eps)m}|h|^{5/14+\eps}\right).
 \end{split}
\end{equation}
Furthermore,
\begin{equation}\label{tw:secondary-bound}
 |c_j(h)|\ll_{q,\eps}
 |h_1|^{-1/2}|h_2|^{-1/3}|h|^{\eps},\qquad 0\le j\le2.
\end{equation}
All three $c_j(h)$ are nonzero, $c_0(h)<0$, and
$c_2(h)=\overline{c_1(h)}$.
\end{theorem}

The weighted summation in \eqref{tw:summed-error} permits degrees
up to $(5/12-\delta)m$ with a power saving error relative to the
family size.

\subsection{Outline of the argument}
We apply the functional equation to $A_3(s,w)$ of
Hong and Fang \cite{HF}. Summing over the variable $D\in\Fq[T]$ introduced by
M\"obius inversion before taking absolute values gives an Euler product controlled
by a Dirichlet $L$-function and its reciprocal. Interpolation with
the bound for $H(M,k)$ and Bochner's tube theorem then continue $A_3$ to
$\Re s>23/28$ on $w=1/2$.

The residues are identified using \cite[(66)--(67)]{DFL}, and one
shift of the coefficient contour gives the two main terms. Working with
$A_3(s,w)$ avoids the truncation parameter of an
approximate functional equation. Section~\ref{sec:twisted} extends
the argument to twists and gives uniform estimates for summation
against short Dirichlet polynomials.

\section{Preliminaries}\label{pre:section}

Throughout, $q$ is a fixed odd prime power satisfying $q\equiv2\pmod3$.
All polynomials in sums and products are monic.  We write $\Aq$ and
$\Aqq$ for the monic polynomials over $\Fq$ and $\Fqq$, respectively,
and $\Hq$ and $\Hqq$ for their squarefree subsets.  Set
$|N|=q^{\deg N}$ and $|F|_2=q^{2\deg F}$, and let $\sigma$ denote
the automorphism of $\Fqq[T]$ induced by $a\mapsto a^q$.

\subsection{Characters, functional equations, and Gauss sums}

We use the cubic characters defined in \cite{DFL,HF}.  Fix an isomorphism $\Omega$ between the complex
cube roots of unity and those in $\Fqq$.  For a prime $\pi\in\Fqq[T]$,
the character $\chi_\pi$ is specified by
\[
 \Omega(\chi_\pi(a))\equiv a^{(q^{2\deg\pi}-1)/3}\pmod\pi
 \qquad (\pi\nmid a),
\]
and is zero on multiples of $\pi$.  Extend this definition
multiplicatively in the modulus.  The same notation denotes restriction
to $\Fq[T]$.  If $F\in\Hqq$, $F\ne1$, has no nonconstant divisor in
$\Fq[T]$, this restriction is a primitive, nonprincipal, even cubic
character of conductor $FF^\sigma$.  Conversely these $F$ parametrize
the primitive cubic characters over $\Fq[T]$, with
\begin{equation}\label{pre:genus}
 \operatorname{genus}(\chi_F)=2\deg F-2.
\end{equation}
For a character over $\Fq[T]$, write
\[
 L_q(w,\chi)=\sum_{N\in\Aq}\frac{\chi(N)}{|N|^w},
 \qquad
 L_q(v,\chi)=\sum_{N\in\Aq}\chi(N)v^{\deg N},
 \qquad v=q^{-w}.
\]
For a nonprincipal character these expressions continue as a polynomial
in $v$.  For the character of conductor one we use
$L_q(v,1)=(1-qv)^{-1}$.

Put $p=\operatorname{char}\Fq$.  For a Laurent series
$x=\sum_jx_jT^j$ define
\[
 e_{q^2}(x)=\exp\!\left(
 \frac{2\pi i}{p}\operatorname{Tr}_{\Fqq/\mathbb F_p}(x_{-1})\right),
 \qquad
 G_{q^2}(V,F)=\sum_{a\bmod F}\chi_F(a)e_{q^2}(aV/F).
\]
The following lemmas use the normalization of \cite{DFL}, with
the base field replaced by $\Fqq$.

\begin{lemma}[{\cite[Lemmas 2.1 and 2.12]{DFL}}]
\label{pre:gauss-identities}
The cubic Gauss sums satisfy
\begin{equation}
 G_{q^2}(V,F)=\overline{\chi_F(V)}G_{q^2}(1,F)
 \qquad ((V,F)=1).\label{pre:twist}
\end{equation}
Cubic reciprocity for monic polynomials gives
$\chi_{F_1}(F_2)=\chi_{F_2}(F_1)$.  Consequently,
\begin{equation}\label{pre:gauss-sieve}
 G_{q^2}(N,DF)=G_{q^2}(N,D)G_{q^2}(ND,F)
 \qquad ((D,F)=1).
\end{equation}
If $(V,F)=1$, the Gauss sum vanishes unless $F$ is squarefree; for
squarefree $F$ one has $|G_{q^2}(1,F)|=q^{\deg F}$.
\end{lemma}
\begin{proof}
Use \cite[(22) and Lemmas~2.1 and~2.12]{DFL};
\eqref{pre:gauss-sieve} follows by multiplicativity and reciprocity.
\end{proof}

\begin{lemma}[{\cite[Lemma 2.3 and Corollary 2.4]{DFL}; \cite[Lemma 2.2]{HF}}]
\label{pre:functional-equation-lemma}
For an admissible $F\ne1$, the even functional equation takes the form
\begin{equation}\label{pre:functional-equation}
 L_q(w,\chi_F)
 =\frac{q^{2w-1}(1-q^{-w})}{1-q^{w-1}}
   \frac{G_{q^2}(1,F)}{|F|_2^w}
   L_q(1-w,\overline{\chi_F}).
\end{equation}
This remains a rational identity for $F=1$ with the convention
$L_q(w,1)=(1-q^{1-w})^{-1}$.
\end{lemma}
\begin{proof}
The root number is $q^{-\deg F}G_{q^2}(1,F)$ by
\cite[proof of Lemma~4.2]{DFL}; see also \cite[Lemma~2.2]{HF}.
The case of conductor $1$ follows by substitution.
\end{proof}

\begin{lemma}[{\cite[Lemma 2.6]{DFL}}]\label{pre:lindelof-lemma}
For every primitive nonprincipal cubic character $\chi$ over $\Fq[T]$
and every $\eps>0$, one has
\begin{equation}\label{pre:lindelof}
 |L_q(w,\chi)|\ll_{q,\eps}q^{\eps\deg\operatorname{cond}(\chi)}
 \qquad (\Re w\ge\tfrac12),
\end{equation}
uniformly in $\Im w$.
\end{lemma}

\begin{lemma}[{\cite[Lemma 4.4]{DFL}}]\label{pre:rational-phase}
If $D\in\Hq$ has degree $m$, then
\[
 G_{q^2}(1,D)=q^m,
 \qquad
 G_{q^2}(N,D)=q^m\quad (N\in\Aq, (N,D)=1).
\]
\end{lemma}
\begin{proof}
The first identity is \cite[Lemma~4.4]{DFL} in the same additive
normalization. Since $\chi_D(N)=1$ for coprime polynomials $N,D\in\Fq[T]$, the second follows from \eqref{pre:twist}.
\end{proof}

\subsection{The generating series of cubic Gauss sum}

Following \cite[Section 4]{HF}, set
\begin{equation}\label{pre:H}
 H(M,k)=\sum_{\substack{F\in\Aqq\\(F,M)=1}}
       \frac{G_{q^2}(M,F)}{|F|_2^k},
 \qquad \Re k>\frac32.
\end{equation}
Thus $H(M,k)=\widetilde\Psi_{q^2}(M,q^{-2k})$ in the notation of
\cite{DFL}.  The following facts are consequences of
\cite[Proposition 3.2, equation (26), Theorem 3.10 and Lemma 3.11]{DFL}.

\Needspace{8\baselineskip}
\begin{lemma}[Consequence of {\cite[Theorem 3.10 and Lemma 3.11]{DFL}}]
\label{pre:gauss-series}
Put $z=q^{-2k}$.  The function $H(M,k)$ is rational in $z$, and
\[
 (1-q^{8-6k})H(M,k)
\]
is holomorphic for $\Re k>2/3$.  More precisely its possible remaining
denominators have the form $1-(q^4z^3)^{\deg\pi}$, with
$\pi\mid M$.  For every compact real interval
$I\subset(2/3,3/2)$ and every $\eps>0$,
\begin{equation}\label{pre:gauss-series-bound}
 |(1-q^{8-6k})H(M,k)|
 \ll_{q,I,\eps}
 |M|_2^{\frac12(\frac32-\Re k)+\eps}
 \qquad (\Re k\in I),
\end{equation}
uniformly in $M$ and $\Im k$.
\end{lemma}
\begin{proof}
The unrestricted series in \cite[equation (26)]{DFL} has denominator
$1-q^8z^3$ after replacing the base field by $\Fqq$.  The
coprimality corrections in \cite[Lemma 3.11]{DFL} have the displayed
local denominators; their zeros satisfy $|z|=q^{-4/3}$.
The same lemma gives \eqref{pre:gauss-series-bound} away from the three
zeros of $1-q^8z^3$.  After multiplying by that factor, the maximum
modulus principle on small disks around these three points gives
the bound there as well.  The disks may be chosen small enough that
the variation in the conductor exponent is absorbed into $\eps$.
The rationality in $z$ gives the asserted uniformity in $\Im k$.
\end{proof}

We use the following specialization of the polar calculation in
\cite{DFL}.

\begin{lemma}[Specialization of {\cite[Proposition 3.1 and Lemmas 3.9 and 4.4]{DFL}}]
\label{pre:rational-residue}
Let $M\in\Aq$ and write uniquely
$M=M_1M_2^2M_3^3$, where $M_1,M_2$ are squarefree and coprime.
Set $\omega=e^{2\pi i/3}$ and $z_j=q^{-8/3}\omega^j$.
Then, for $j=0,1,2$,
\begin{align}
 &\left.(1-q^{8-6k})H(M,k)\right|_{q^{-2k}=z_j}
 \label{pre:polar-coefficient}\\
 &\quad=
 \frac{\mathbf1_{M_2=1}}{\zeta_{q^2}(2)}
 (1+q^{1/3}\omega^j)
 q^{-\deg M_1/3}\omega^{-j\deg M_1}
 \prod_{\substack{\pi\mid M\\\pi\text{ prime over }\Fqq}}
 (1+|\pi|_2^{-1})^{-1},\notag
\end{align}
where $\zeta_{q^2}(2)=(1-q^{-2})^{-1}$.
Equivalently, the last product equals $\prod_{P\mid M}\beta_P$,
over primes $P\in\Fq[T]$, where for $d=\deg P$,
\begin{equation}\label{pre:beta}
 \beta_P=
 \begin{cases}
 (1+q^{-2d})^{-1},&d\text{ odd},\\
 (1+q^{-d})^{-2},&d\text{ even}.
 \end{cases}
\end{equation}
\end{lemma}
\begin{proof}
Apply \cite[Proposition~3.1 and Lemma~3.9]{DFL} over $\Fqq$.
Put $a=\deg M_1$, and let $[b]_3\in\{0,1,2\}$ denote the
residue of $b$ modulo $3$. In the present normalization,
\[
 \rho(1,0)=1,\qquad \rho(1,1)=q^3,\qquad \rho(1,2)=0.
\]
Here $\tau(\chi_3)=q$ by \cite[Lemma~4.4]{DFL}.
For fixed $M$, the error in the cited formula for the coefficient of degree $d$
is $O_{q,M}(q^{2d})$, so its generating function is holomorphic
near all $z_j$. Summing the main term over the three degree classes
gives denominator $1-q^8z^3$ and, at $z=z_j$, the factor
\[
 \sum_{i=0}^2\omega^{ji}q^{-8[i+a]_3/3}\rho(1,[i+a]_3)
 =\omega^{-ja}(1+q^{1/3}\omega^j).
\]
Finally Lemma~\ref{pre:rational-phase} gives
$G_{q^2}(1,M_1)|M_1|_2^{-2/3}=q^{-a/3}$.
A prime in $\Fq[T]$ of odd degree remains prime over $\Fqq$, whereas
one of even degree splits into two primes of half its degree.
This proves \eqref{pre:beta} and the lemma.
\end{proof}

\subsection{The tube theorem}

For an open subset $U\subset\mathbb R^2$, write
$T(U)=\{(s,w)\in\mathbb C^2:(\Re s,\Re w)\in U\}$.
We use the following form of Bochner's tube theorem \cite{Bochner}.

\begin{theorem}[Bochner]\label{pre:bochner}
If $U$ is open and connected, every holomorphic function on $T(U)$
extends uniquely to a holomorphic function on
$T(\operatorname{conv}U)$.
\end{theorem}

The generating variables in a Dirichlet series over a function field are
$u=q^{-2s}$ and $v=q^{-w}$.  Their domains are described by the
moduli $|u|$ and $|v|$, whereas Theorem~\ref{pre:bochner} concerns
the real parts of $(s,w)$.  The following consequence makes the
passage between the two descriptions precise.

\begin{lemma}[Tube theorem in the generating variables]
\label{pre:tube-generating}
Let $q>1$ and let $U\subset\mathbb R^2$ be open and connected. Set
\begin{equation}\label{pre:log-domain}
 \mathcal D(U)=\left\{(u,v)\in(\mathbb C^\times)^2:
 \left(-\frac{\log|u|}{2\log q},
       -\frac{\log|v|}{\log q}\right)\in U\right\}.
\end{equation}
Every holomorphic function $F$ on $\mathcal D(U)$ has a unique
holomorphic continuation to
$\mathcal D(\operatorname{conv}U)$.
Equivalently, the holomorphic continuation of
$f(s,w)=F(q^{-2s},q^{-w})$ to $T(\operatorname{conv}U)$ preserves
the periods
\begin{equation}\label{pre:periods}
 f\left(s+\frac{\pi i}{\log q},w\right)
 =f(s,w)
 =f\left(s,w+\frac{2\pi i}{\log q}\right).
\end{equation}
If $F$ is meromorphic and $P(u,v)F(u,v)$ is holomorphic on
$\mathcal D(U)$, where $P$ is a nonzero Laurent polynomial, then
$F$ continues meromorphically to $\mathcal D(\operatorname{conv}U)$,
with possible poles only on $P(u,v)=0$.
\end{lemma}

\begin{proof}
The map $(s,w)\mapsto(q^{-2s},q^{-w})$ takes $T(U)$ onto
$\mathcal D(U)$ and is locally biholomorphic. Its fibres differ
by integer multiples of $(\pi i/\log q,0)$ and
$(0,2\pi i/\log q)$. The pullback $f$ is holomorphic on $T(U)$
and satisfies \eqref{pre:periods}. By Theorem~\ref{pre:bochner},
it extends uniquely to a holomorphic function $\widetilde f$ on
$T(\operatorname{conv}U)$.

Both $T(U)$ and $T(\operatorname{conv}U)$ are invariant under
these imaginary translations. Thus
\[
 \widetilde f\left(s+\frac{\pi i}{\log q},w\right)
       -\widetilde f(s,w)
\]
is holomorphic on the connected larger tube and vanishes on $T(U)$.
The identity theorem makes it identically zero. The same argument
applies to the period in $w$. Consequently $\widetilde f$ is
constant on every fibre of the exponential map and defines a
function on $\mathcal D(\operatorname{conv}U)$.
Local branches of the logarithm show that this function is
holomorphic. Uniqueness follows by pulling back to the larger tube.
Finally, apply the holomorphic assertion to $PF$ and divide its
continuation by $P$ to obtain the meromorphic assertion.
\end{proof}

\begin{remark}\label{pre:tube-application}
For a union of initial regions, the local formulas must first agree
on their overlaps and define one pole-free holomorphic function
on a connected domain. Lemma~\ref{pre:tube-generating} then takes
the convex hull in the logarithmic coordinates
\eqref{pre:log-domain}. It does not take the ordinary convex hull
of the domains in $(u,v)$.

The domains in the lemma exclude the coordinate axes. In the
central slice used below, the original power series is holomorphic
at $u=0$ for $v=q^{-1/2}$. Agreement near $u=0$ therefore fills in
that point after continuation. Moreover, periodicity bounds each
fixed continued holomorphic function on every tube with compact
real base inside $\operatorname{conv}U$: it suffices to restrict
the imaginary parts to a compact period rectangle.
\end{remark}

\section{The generating series and the proof of Theorem~\ref{thm:main}}

\subsection{The generating series and its initial regions}
Following Hong and Fang \cite{HF}, define
\begin{align}
A_3(s,w)&=\sum_{\substack{F\in\Hqq\\ F\text{ has no nonconstant divisor in }\Fq[T]}}
\frac{L_q(w,\chi_F)}{|F|_2^s},\label{eq:A3dir}\\
A_3(u,v)&=\sum_{\substack{F\in\Hqq\\ F\text{ has no nonconstant divisor in }\Fq[T]}}
L_q(v,\chi_F)u^{\deg F}.\label{eq:A3gen}
\end{align}
Here the notation $L_q(v,\chi)=\sum_{N\in\Aq}\chi(N)v^{\deg N}$ denotes the same $L$-function in the variable $v$. Thus \eqref{eq:variables} is understood throughout. We include $F=1$, with $L_q(w,1)=(1-q^{1-w})^{-1}$ and $L_q(v,1)=(1-qv)^{-1}$. This affects only the constant coefficient in $u$.

The family parametrization and Cauchy's formula give, for $0<r<q^{-2}$,
\begin{equation}\label{eq:perron}
\sumfamily L_q\!\left(\frac12,\chi\right)
=\frac1{2\pi i}\oint_{|u|=r}A_3(u,q^{-1/2})\frac{du}{u^{g/2+2}}.
\end{equation}
There is no factor $(1-u)^{-1}$, since the genus is fixed.

We use real bases $R_j$ and write $S_j=T(R_j)$ for their tubes in the Dirichlet variables. The two initial regions corresponding to the regions in the generating variables in \cite{HF} are
\begin{align}
R_1&=\{(\sigma,\tau):\sigma>1/2,\ \tau>1\},\label{eq:R1}\\
R_2&=\{(\sigma,\tau):\sigma>1,\ \sigma+\tau>3/2\}.\label{eq:R2}
\end{align}
Their argument gives holomorphy of $(1-q^{2-2s})A_3(s,w)$ on $S_1$; see \cite[Section~3]{HF}. On $S_2$ the series \eqref{eq:A3dir} converges normally apart from the term with conductor $1$ at $q^{w-1}=1$. For $\Re w\ge1/2$ this follows from the Lindel\"of bound, and for $\Re w<1/2$ from the functional equation followed by the same bound. Consequently
\begin{equation}\label{eq:commonproduct}
(1-q^{2-2s})(1-q^{w-1})(1-q^{8-6(s+w)})A_3(s,w)
\end{equation}
is holomorphic on both initial tubes. We will continue this common product.

\subsection{The dual representation and an exact first difference}
Set $k=s+w$ and $t=q^{w-1}$. The functional equation, M\"obius inversion over $\Fq[T]$ and multiplicativity of Gauss sums give the dual representation of Hong and Fang \cite{HF}:
\begin{equation}\label{eq:dual}
A_3(s,w)=\frac{q^{2w-1}(1-q^{-w})}{1-q^{w-1}}
\sum_{\substack{N,D\in\Aq\\(N,D)=1}}
|N|^{w-1}\mu(D)G_{q^2}(N,D)|D|_2^{-k}H(ND,k).
\end{equation}
Here $\mu$ is over $\Fq[T]$, and the opened sums converge absolutely
for $\Re k>3/2$ and $\Re w<0$. The derivation uses
\eqref{pre:gauss-sieve}; for $F=1$, the functional equation is
interpreted as the rational identity giving $(1-q^{1-w})^{-1}$.

For $n,m\ge0$, set
\begin{equation}\label{eq:Bnm}
B_{nm}(k)=\sum_{\substack{N,D\in\Aq\\\deg N=n,\ \deg D=m\\(N,D)=1}}
\mu(D)G_{q^2}(N,D)(1-q^{8-6k})H(ND,k),
\end{equation}
and put $\Delta B_{nm}=B_{nm}-qB_{n-1,m}$, with $B_{-1,m}=0$. Since
\[
q^{2w-1}\frac{1-q^{-w}}{1-q^{w-1}}=-\frac{t(1-qt)}{1-t},
\]
we have the exact identity
\begin{equation}\label{eq:diff}
(1-q^{8-6k})A_3(s,w)
=-\frac{t}{1-t}\sum_{n,m\ge0}\Delta B_{nm}(k)t^nq^{-2km}.
\end{equation}
The subtraction $qB_{n-1,m}$ is thus supplied by the original functional equation.

\subsection{Averaging over \texorpdfstring{$D$}{D}}
For $\Re k>3/2$ we open $H(ND,k)$. Its nonsquarefree terms vanish because its summation moduli are coprime to $ND$. Every remaining modulus has a unique decomposition
\[
F=EG,\qquad E\in\Hq,\quad G\in\Hqq,
\]
where $E$ is its maximal divisor in $\Fq[T]$ and $G$ has no nonconstant divisor in $\Fq[T]$. The polynomials $E$ and $GG^\sigma$ are coprime. Put $\ell=\deg G$ and $\chi=\overline{\chi_G}|_{\Fq[T]}$, which is primitive and nonprincipal of conductor $GG^\sigma$ when $G\ne1$. When $G=1$, set $\chi=1$.

To sum over $D$, define
\begin{equation}\label{eq:SF}
\mathcal S_F(x,y)=(1-qx)
\sum_{\substack{D\in\Hq\\(D,F)=1}}
\mu(D)\chi(D)y^{\deg D}
\sum_{\substack{N\in\Aq\\(N,D)=1}}
\overline{\chi_F(N)}x^{\deg N}.
\end{equation}
Write
\begin{equation}\label{eq:RF}
R_F(x)=(1-qx)L_q(x,\chi)
\prod_{P\mid E}(1-\chi(P)x^{\deg P}).
\end{equation}
For $G=1$, the convention $L_q(x,1)=(1-qx)^{-1}$ gives
$R_F(x)=\prod_{P\mid E}(1-x^{\deg P})$ by exact cancellation.
For a prime $P\in\Fq[T]$ of degree $d$, choosing $P\mid D$ contributes $-\chi(P)y^d(1-\chi(P)x^d)$. Therefore
\begin{align}
\mathcal S_F(x,y)
&=R_F(x)\prod_{P\nmid E}(1-\chi(P)y^d+\chi(P)^2(xy)^d)\label{eq:SFproduct}\\
&=\frac{R_F(x)}{L_q(y,\chi)}
\prod_{P\mid E}(1-\chi(P)y^d)^{-1}
\prod_{P\nmid E}\left(1+\frac{\chi(P)^2(xy)^d}{1-\chi(P)y^d}\right).
\label{eq:SFfactor}
\end{align}
Primes dividing $GG^\sigma$ are automatically omitted through the zero values of $\chi$. Thus \eqref{eq:SFproduct} retains all coprimality conditions. Its last product converges absolutely and locally uniformly when
\begin{equation}\label{eq:correctionregion}
|y|<1,\qquad q|xy|<1.
\end{equation}
The bound on every smaller closed polydisk is independent of $E$ and $G$, since the absolute local corrections are at most $|xy|^d/(1-|y|)$.

Lemma~\ref{pre:rational-phase} and the twisting identity \eqref{pre:twist} yield
\begin{equation}\label{eq:opened}
\begin{split}
\Delta B_{nm}(k)
={}&(1-q^{8-6k})q^m
\sum_{F\in\Hqq}G_{q^2}(1,F)q^{-2k\deg F}\\
&\quad\times\frac1{(2\pi i)^2}
\oint_{|x|=r_x}\oint_{|y|=r_y}
\mathcal S_F(x,y)\frac{dy\,dx}{x^{n+1}y^{m+1}}.
\end{split}
\end{equation}
Initially the circles are small and $\Re k>3/2$. The two integrals select the finite degrees before estimating the remaining sum over $F$.

We first bound the reciprocal $L$-function on a circle strictly
inside its zero-free disk.

\begin{lemma}\label{lem:inverseL}
For $0<\delta<1/2$ and every $\eps>0$, uniformly for the above nonprincipal $\chi$ and all $|y|=q^{-1/2-\delta}$,
\begin{equation}\label{eq:inverseL}
|L_q(y,\chi)|^{-1}\ll_{q,\delta,\eps}q^{\eps\ell}.
\end{equation}
\end{lemma}
\begin{proof}
Choose the logarithm vanishing at $y=0$ in $|y|<q^{-1/2}$.
The Euler product, using $\sum_{\deg f=j}\Lambda(f)=q^j$, and
the Riemann hypothesis give
\[
 |\log L_q(y,\chi)|\le
 \sum_{j\ge1}\frac{\min\{q^j,(2\ell+O(1))q^{j/2}+1\}}
 {j q^{(1/2+\delta)j}}.
\]
The second bound includes the root $1$ from the even character.
Splitting the sum at $j=\lfloor\log_q(\ell+2)\rfloor$ bounds it by
$O_{q,\delta}((\ell+2)^{1-\delta}+\log(\ell+2))=o(\ell)$,
uniformly in the argument of $y$. Exponentiation proves
\eqref{eq:inverseL}, with bounded conductors absorbed in the constant.
\end{proof}

Use the full coefficient torus
\begin{equation}\label{eq:torus}
|x|=q^{-1/2},\qquad |y|=q^{-1/2-\delta}.
\end{equation}
The Lindel\"of bound gives $|L_q(x,\chi)|\ll_{q,\eps}q^{\eps\ell}$. The finite factors over $P\mid E$ in \eqref{eq:RF} and \eqref{eq:SFfactor} contribute $\ll_{q,\delta,\eps}q^{\eps\deg E}$. Indeed, there are only finitely many primes below a fixed degree cutoff, while above a sufficiently large cutoff the logarithm of each absolute local cost is at most $\eps\deg P\log q$.

Since $q|xy|=q^{-\delta}<1$, Lemma~\ref{lem:inverseL} and \eqref{eq:SFfactor} imply
\begin{equation}\label{eq:SFbound}
|\mathcal S_F(x,y)|\ll_{q,\delta,\eps}q^{\eps\deg F}
\end{equation}
on \eqref{eq:torus}. When $G=1$, use $L_q(y,1)^{-1}=1-qy$ and the cancellation in \eqref{eq:RF}; the same estimate holds. For each $F$ the factorization is holomorphic on a neighborhood of the closed polydisk bounded by \eqref{eq:torus}. Hence double Cauchy extraction gives
\begin{equation}\label{eq:doubleCauchy}
\left|\frac1{(2\pi i)^2}\oint\oint
\mathcal S_F(x,y)\frac{dy\,dx}{x^{n+1}y^{m+1}}\right|
\ll_{q,\delta,\eps}q^{n/2+(1/2+\delta)m+\eps\deg F}.
\end{equation}

For $F\in\Hqq$ of degree $h$, we have $|G_{q^2}(1,F)|=q^h$, and there are at most $q^{2h}$ choices of $F$. Substituting \eqref{eq:doubleCauchy} into \eqref{eq:opened} gives
\[
|\Delta B_{nm}(k)|
\ll |1-q^{8-6k}|q^{n/2+(3/2+\delta)m}
\sum_{h\ge0}q^{(3-2\Re k+\eps)h}.
\]
On $\Re k=3/2+\kappa$, choose the loss in $h$ smaller than $2\kappa$ and then take $\delta$ as small as the desired loss in $m$. We obtain
\begin{equation}\label{eq:rightendpoint}
|\Delta B_{nm}(3/2+\kappa+i\theta)|
\ll_{q,\kappa,\eps}q^{n/2+(3/2+\eps)m}
\quad(\theta\in\mathbb R).
\end{equation}
Summing over $D$ before taking absolute values gives the improved bound \eqref{eq:rightendpoint}.

At $a=2/3+\kappa$, Lemma~\ref{pre:gauss-series}, applied to $M=ND$, gives
\begin{equation}\label{eq:leftendpoint}
|\Delta B_{nm}(a+i\theta)|
\ll_{q,\kappa,\eps}q^{(5/2-a+\eps)n+(7/2-a+\eps)m}.
\end{equation}
Here we counted $q^n$ choices for $N$ and at most $q^m$ for $D$, used $|G_{q^2}(N,D)|\le q^m$, and absorbed the factor from $qB_{n-1,m}$ in the constant.

\begin{proposition}\label{prop:averaged}
For every compact interval $I\subset(2/3,3/2)$ and every $\eps>0$,
\begin{equation}\label{eq:averaged}
|\Delta B_{nm}(k)|
\ll_{q,I,\eps}
q^{(29/10-8\Re k/5+\eps)n+(39/10-8\Re k/5+\eps)m},
\end{equation}
uniformly for $n,m\ge0$, $\Re k\in I$ and $\Im k\in\mathbb R$.
\end{proposition}
\begin{proof}
Put $z=q^{-2k}$. The pole-free function $(1-q^{8-6k})H(M,k)$ is rational in $z$ and holomorphic for $\Re k>2/3$. The same is true of the finite sum $\Delta B_{nm}$. Apply Hadamard's three circles theorem on
\[
q^{-2b}\le|z|\le q^{-2a},\qquad a=2/3+\kappa,\quad b=3/2+\kappa.
\]
Both boundary estimates hold on full circles. With zero displacement, the exponent pairs are $(11/6,17/6)$ at $2/3$ and $(1/2,3/2)$ at $3/2$. Their linear interpolants are exactly the exponents in \eqref{eq:averaged}. The actual displaced interpolants differ by $O_I(\kappa)$, in addition to the arbitrarily small endpoint losses. Choose $\kappa$ and these losses after prescribing the final $\eps$, with $a<\inf I$. This proves the result. No estimate on the singular line $\Re k=2/3$ or uniformity as $\kappa\downarrow0$ is asserted.
\end{proof}

\subsection{Convergence region and convex hull}
After multiplication by $t^nq^{-2km}$, the two degree exponents in \eqref{eq:averaged} become
\begin{equation}\label{eq:weightedexponents}
\frac{19}{10}-\frac85\Re s-\frac35\Re w+\eps,
\qquad
\frac{39}{10}-\frac{18}{5}\Re(s+w)+\eps.
\end{equation}
Consequently the pole-free series in \eqref{eq:diff} converges normally on the tube over
\begin{equation}\label{eq:R3D}
R_{3,D}=\left\{(\sigma,\tau):
16\sigma+6\tau>19,\quad \frac{13}{12}<\sigma+\tau<\frac32\right\}.
\end{equation}
On every compact subset choose $\eps$ smaller than both positive margins in \eqref{eq:weightedexponents}; the resulting double geometric majorant proves normal convergence.

We give the connection to the initial function explicitly. The tubes over $R_1$ and $R_2$ overlap. Inside $R_2$, the range $\Re k>3/2$, $\Re w<0$ is a region of absolute convergence for the opened dual identity. The right endpoint estimate \eqref{eq:rightendpoint}, locally uniformly for $\Re k>3/2$, continues the same series normally to the tube over
\begin{equation}\label{eq:R3R}
R_{3,R}=\{(\sigma,\tau):\sigma+\tau>3/2,\ \tau<1/2\}.
\end{equation}
There is also an open overlap across $\Re k=3/2$. To see this, retain a fixed small displacement $\kappa>0$ in the exact interpolant between \eqref{eq:leftendpoint} and \eqref{eq:rightendpoint}. Its right edge is above $3/2$. At $\Re k=3/2$ and any fixed $\Re w<1/2$, the limiting weighted exponents are $\Re w-1/2<0$ and $-3/2<0$. For sufficiently small displacement both remain negative on an open neighborhood crossing this line. These neighborhoods connect the tubes over \eqref{eq:R3D} and \eqref{eq:R3R}.

On their absolute overlaps the definitions agree, so uniqueness glues \eqref{eq:commonproduct} to one holomorphic function on the connected union of these tubes. The factor $1-q^{w-1}$ removes the denominator in \eqref{eq:diff}, $1-q^{8-6(s+w)}$ removes the poles of $H(ND,s+w)$, and $1-q^{2-2s}$ clears the known main divisor.

Bochner's tube theorem and Lemma~\ref{pre:tube-generating} give the
following continuation.

\begin{proposition}\label{prop:continuation}
The product \eqref{eq:commonproduct} has a holomorphic continuation to
$S_4=T(R_4)$, where
\begin{equation}\label{eq:R4}
R_4=\left\{(\sigma,\tau):
\sigma>\frac12,\quad\sigma+\tau>\frac{13}{12},\quad14\sigma+9\tau>16\right\}.
\end{equation}
Equivalently, the product
\[
 (1-q^2u)(1-q^{-1}v^{-1})(1-q^8u^3v^6)A_3(u,v)
\]
is holomorphic on
\begin{equation}\label{eq:R4-generating}
 \mathcal D(R_4)=\left\{(u,v)\in(\mathbb C^\times)^2:
 |u|<q^{-1},\quad |u|^{1/2}|v|<q^{-13/12},\quad
 |u|^7|v|^9<q^{-16}\right\}.
\end{equation}
In particular,
\begin{equation}\label{eq:centrals}
(1-q^{2-2s})(1-q^{5-6s})A_3(s,1/2)
\quad\text{is holomorphic for }\Re s>\frac{23}{28},
\end{equation}
or equivalently on $0<|u|<q^{-23/14}$ in the variable $u$.
On the divisor $\Re(s+w)=4/3$, the region \eqref{eq:R4} contains
all points with $\Re w<8/15$.
\end{proposition}

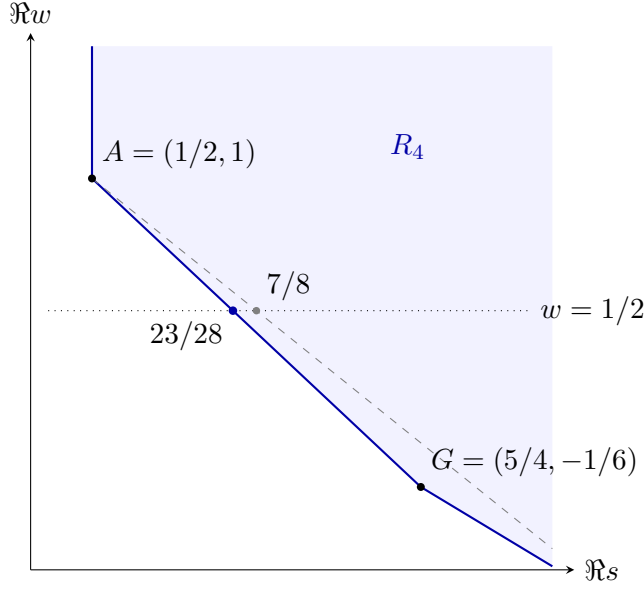
\begin{figure}[tb]
\centering
\begin{tikzpicture}[x=5.8cm,y=3.5cm,>=stealth]
\fill[blue!5] (.5,1) -- (.5,1.5) -- (1.55,1.5) -- (1.55,-.466667) -- (1.25,-.166667) -- cycle;
\draw[->] (.36,-.48)--(1.6,-.48) node[right]{$\Re s$};
\draw[->] (.36,-.48)--(.36,1.55) node[above]{$\Re w$};
\draw[thick,blue!65!black] (.5,1.5)--(.5,1)--(1.25,-.166667)--(1.55,-.466667);
\draw[dashed,gray] (.5,1)--(1.55,-.4);
\draw[dotted] (.4,.5)--(1.5,.5) node[right]{$w=1/2$};
\fill (.5,1) circle(1.5pt) node[above right]{$A=(1/2,1)$};
\fill (1.25,-.166667) circle(1.5pt) node[above right]{$G=(5/4,-1/6)$};
\fill[blue!65!black] (.82142857,.5) circle(1.6pt);
\node[below left] at (.82142857,.5) {$23/28$};
\fill[gray] (.875,.5) circle(1.4pt);
\node[above right] at (.875,.5) {$7/8$};
\node[blue!65!black] at (1.22,1.12) {$R_4$};
\end{tikzpicture}
\caption{The relevant convex hull in the Dirichlet variables. The solid segment is $14\Re s+9\Re w=16$; the dashed line is the previous boundary $4\Re s+3\Re w=5$. The new central boundary lies to the left of $5/6$ by $1/84$.}
\label{fig:hull}
\end{figure}

\subsection{The residues of \texorpdfstring{$A_3(s,w)$}{A3(s,w)}}
\label{subsec:full-residues}

We identify the residues of $A_3(s,w)$ with the
arithmetic factors evaluated in \cite[pp.~44--46, (66)--(67)]{DFL}.
The required steps are to justify taking residues through the $N,D$
sum and to match its variables with those of the cited formulas.
Put $z=q^{-2(s+w)}$ and, for $\zeta^3=1$, write
\[
 z_\zeta=q^{-8/3}\zeta^{-1},\qquad
 u_\zeta(w)=q^{-8/3+2w}\zeta^{-1}.
\]
Thus $z=u q^{-2w}$, and the normalized residue $-\Res A_3/u_\zeta(w)$
is unchanged if the residue is instead taken in $z$ and divided by
$z_\zeta$.  We first work with $\Re w<0$.

Write uniquely
\[
 N=N_1N_2^2N_3^3,
 \qquad N_1,N_2\in\Hq,\quad (N_1,N_2)=1,
\]
where $N_3$ is monic and need not be coprime to $N_1N_2$.
If $D\in\Hq$ and $(D,N)=1$, the cubefree factors of $ND$ are
$N_1D$ and $N_2$.
Regard $H(ND,k)$ as a function of $z=q^{-2k}$. Then
Lemma~\ref{pre:rational-residue} gives, for each $\zeta^3=1$,
\begin{equation}\label{eq:H-principal-parts}
\begin{split}
 -\frac1{z_\zeta}\Res_{z=z_\zeta}H(ND,k)
 ={}&\frac{\mathbf 1_{N_2=1}\,
 q^{-(\deg N_1+\deg D)/3}}{3\zeta_{q^2}(2)}
 (1+q^{1/3}\zeta^2)\zeta^{\deg N_1+\deg D}\\
 &\times\prod_{\substack{\pi\text{ prime in }\Fqq[T]\\\pi\mid ND}}
       (1+|\pi|_2^{-1})^{-1}.
\end{split}
\end{equation}

The residues of $A_3(s,w)$ can be evaluated as follows.
The averaged estimate for $\Delta B_{nm}$ used in
Proposition~\ref{prop:continuation} makes the pole-free series
\eqref{eq:diff} normally convergent near each $z_\zeta$ when
$\Re w<0$: at $\Re(s+w)=4/3$ its two geometric exponents are
$\Re w-7/30$ and $-9/10$, before arbitrarily small losses.
Consequently the residue of $A_3(s,w)$ can be computed term by term in the blocks of fixed degree.  After \eqref{eq:H-principal-parts} is inserted, the resulting
individual $N,D$ residue series is absolutely convergent for
$\Re w<1/3$.  In fact its absolute value is majorized, up to a fixed
constant, by
\[
 \sum_{N_1\in\Hq}\sum_{N_3\in\Aq}
 |N_1|^{\Re w-4/3}|N_3|^{3\Re w-3}
 \sum_{D\in\Hq}|D|^{-2},
\]
which converges for $\Re w<1/3$.
We may thus undo the first difference by its exact generating identity.
This recovers the prefactor in \eqref{eq:dual}; it does not assume normal
convergence of \eqref{eq:dual} at the pole.

At $z=z_\zeta$ the powers depending on $D$ in \eqref{eq:dual} and
\eqref{eq:H-principal-parts} combine as
\[
 q^{\deg D}z_\zeta^{\deg D}
 q^{-\deg D/3}\zeta^{\deg D}=q^{-2\deg D}.
\]
Moreover $\zeta^{\deg N_1}=\zeta^{\deg N}$ whenever $N_2=1$.
It remains to evaluate the absolutely convergent arithmetic sum
\begin{align}
 &\sum_{\substack{N\in\Aq\\N_2=1}}
 |N|^{w-1}q^{-\deg N_1/3}\zeta^{\deg N}
 \prod_{\pi\mid N}(1+|\pi|_2^{-1})^{-1}
 \sum_{\substack{D\in\Hq\\(D,N)=1}}
 \mu(D)|D|^{-2}\prod_{\pi\mid D}(1+|\pi|_2^{-1})^{-1}.
 \label{eq:residue-arithmetic-sum}
\end{align}
Products over $\pi$ in this formula are over primes of $\Fqq[T]$.

Put $X=q^{w-1/2}\zeta$. The definitions of
$\mathcal J_{\mathrm{nK}}$ and $\mathcal H_{\mathrm{nK}}$ in
\cite[pp.~44--45]{DFL} identify
\eqref{eq:residue-arithmetic-sum} with
$\mathcal J_{\mathrm{nK}}(1)\mathcal H_{\mathrm{nK}}(X,1)$:
their polynomial $f$ is our $N$, and their second auxiliary
variable is set equal to $1$. In their notation the corresponding
factors are $C_P(1)=1$ and $D_P(1)=1+2/|P|$.
We now apply \cite[(66)--(67)]{DFL} directly to obtain
\begin{equation}\label{eq:K-identification}
 \mathcal J_{\mathrm{nK}}(1)\mathcal H_{\mathrm{nK}}(X,1)
 =\frac{\mathcal J_{\mathrm{nK}}(1)\mathcal B_{\mathrm{nK}}(X,1)}
        {1-q^{1/6}X}
 =\frac{\mathcal K_{\mathrm{nK}}(X)}{1-q^{1/6}X}.
\end{equation}
Here $\mathcal K_{\mathrm{nK}}$ is the function already recorded
in \eqref{eq:KnK}; its absolute convergence for $|X|<q^{1/6}$
is given in \cite[p.~46]{DFL}.

Combining \eqref{eq:dual} and
\eqref{eq:H-principal-parts}--\eqref{eq:K-identification} now gives
\begin{align}
 -\frac{1}{u_\zeta(w)}\Res_{u=u_\zeta(w)}A_3(u,q^{-w})
 &=\frac{q^{2w-1}(1-q^{-w})}{1-q^{w-1}}
   \frac{1+q^{1/3}\zeta^2}{3\zeta_{q^2}(2)}
   \frac{\mathcal K_{\mathrm{nK}}(q^{w-1/2}\zeta)}
        {1-q^{w-1/3}\zeta}.
 \label{eq:full-residue-general-w}
\end{align}
Initially this identity holds for $\Re w<0$.
By Proposition~\ref{prop:continuation}, its left side is meromorphic
on the connected parameter domain $\Re w<8/15$ along the polar
divisor. The right side is meromorphic there because
$|q^{w-1/2}|<q^{1/6}$. The identity theorem extends the equality
to this domain, including $w=1/2$. Intersections with the main
divisor, such as $w=1/3$ on the real branch, are interpreted
meromorphically.

At $w=1/2$ the first factor on the right side of
\eqref{eq:full-residue-general-w} is $1$.
With $\omega=e^{2\pi i/3}$ and $u_j=q^{-5/3}\omega^j$, we obtain
\begin{equation}
 c_j=-\frac{1}{u_j}\Res_{u=u_j}A_3(u,q^{-1/2})
 =\frac{(1+q^{1/3}\omega^{-2j})
          \mathcal K_{\mathrm{nK}}(\omega^{-j})}
        {3\zeta_{q^2}(2)(1-q^{1/6}\omega^{-j})},
 \qquad j=0,1,2.
 \label{eq:explicit-cj}
\end{equation}
\begin{remark}\label{rem:secondary-residues}
Each product value in \eqref{eq:explicit-cj} is nonzero.
For each prime $P$, put $R=|P|$, $d=\deg P$ and
$Z=X^dR^{-5/6}$. For $X^3=1$, one has $|Z|=R^{-5/6}<1$.
The odd local factors in \eqref{eq:KnK} cannot vanish.
For an even local factor, its second numerator multiplied by $R$
is $R+2+RZ-2R^{-3/2}$, and
\[
 R+2-2R^{-3/2}>R>|RZ|,
\]
so its remaining numerator cannot vanish either.
Absolute convergence then proves
$\mathcal K_{\mathrm{nK}}(X)\ne0$ for $X^3=1$.
At $X=1$ every local factor is positive, and hence
\begin{equation}
 c_0=-\frac{(1+q^{1/3})\mathcal K_{\mathrm{nK}}(1)}
              {3\zeta_{q^2}(2)(q^{1/6}-1)}<0.
 \label{eq:c0-negative}
\end{equation}
The periodic combination occurring in Theorem~\ref{thm:main}
is therefore not identically zero, by the invertibility of the
Fourier transform on three points.  This argument does not require, or
assert, its nonvanishing in every residue class of the genus.
The relation $c_2=\overline{c_1}$ also shows that this combination is real.
\end{remark}

\subsection{Proof of Theorem~\ref{thm:main}}

By Proposition~\ref{prop:continuation} and
Lemma~\ref{pre:tube-generating},
\[
 (1-q^2u)(1-q^5u^3)A_3(u,q^{-1/2})
\]
is holomorphic in $|u|<q^{-23/14}$.
The lemma gives the continuation on the punctured
disc, and the original power series supplies the holomorphic
extension at $u=0$, as in Remark~\ref{pre:tube-application}.
Thus the continued function agrees with the original generating
series wherever \eqref{eq:perron} is initially evaluated.
For completeness, if $m=g/2+1$ and $r=q^{-2\sigma}$ with
$\sigma>1$, the change of variables gives exactly
\begin{equation}\label{eq:perron-logarithmic}
 \frac{1}{2\pi i}\oint_{|u|=r}
       A_3(u,q^{-1/2})\frac{du}{u^{m+1}}
 =\frac{\log q}{\pi i}
   \int_{\sigma}^{\sigma+\pi i/\log q}
       A_3(s,1/2)q^{2ms}\,ds.
\end{equation}
The circle on the left is positively oriented; the upward period
segment on the right maps to a clockwise circle, cancelling the
minus sign in $du/u=-2\log q\,ds$. Hence the passage through the
Dirichlet variables leaves the coefficient normalization unchanged.
Choose $0<\eps<1/84$ and set $R=q^{-2(23/28+\eps)}$,
so that $q^{-2}<q^{-5/3}<R<q^{-23/14}$.
For $m=g/2+1$, Perron's formula and the outward contour shift give
\begin{align}
 \sum_{\substack{\chi\text{ primitive cubic}\\\operatorname{genus}(\chi)=g}}
 L_q(1/2,\chi)
 &=\frac1{2\pi i}\int_{|u|=R}
       A_3(u,q^{-1/2})\frac{du}{u^{m+1}}
   -q^{2m+2}\Res_{u=q^{-2}}A_3(u,q^{-1/2})
 \nonumber\\[-1mm]
 &\hspace{10mm}
   -\sum_{j=0}^2u_j^{-m-1}\Res_{u=u_j}A_3(u,q^{-1/2}).
 \label{eq:final-contour}
\end{align}
The minus signs follow because the outer integral minus the inner
integral is the sum of the crossed residues.

The known principal residue from \cite{HF}, equivalently the main
constant in \cite[Theorem~1.1]{DFL}, is
\[
 -q^2\Res_{u=q^{-2}}A_3(u,q^{-1/2})
 =\frac{\zeta_q(3/2)}{\zeta_q(3)}
   \mathcal A_{\mathrm{nK}}(q^{-2},q^{-3/2}).
\]
Its contribution in \eqref{eq:final-contour} is therefore
$q^{g+2}\zeta_q(3/2)\mathcal A_{\mathrm{nK}}(q^{-2},q^{-3/2})/
\zeta_q(3)$.
By \eqref{eq:explicit-cj}, the other three residues contribute
\[
 \sum_{j=0}^2c_j u_j^{-m}
 =q^{5g/6+5/3}
   \sum_{j=0}^2c_j\omega^{-j(g/2+1)}.
\]
In particular the genus dependence of the secondary coefficient is
periodic modulo $6$.

Finally the pole-free product $(1-q^2u)(1-q^5u^3)A_3(u,q^{-1/2})$ is bounded on the fixed
circle $|u|=R$, and
\[
 |1-q^2u|\ge q^2R-1>0,\qquad
 |1-q^5u^3|\ge q^5R^3-1>0.
\]
Thus $\sup_{|u|=R}|A_3(u,q^{-1/2})|\ll_{q,\eps}1$, and the integral
in \eqref{eq:final-contour} is
\[
 O_{q,\eps}(R^{-m})
 =O_{q,\eps}\bigl(q^{(23/28+\eps)g}\bigr).
\]
Together with \eqref{eq:explicit-cj} and \eqref{eq:c0-negative}, this
proves Theorem~\ref{thm:main}.

\section{Twisted first moments}\label{sec:twisted}

\subsection{The arithmetic factors of the twisted formula}
For a monic $h\in\Aq$, write uniquely
\[
 h=h_1h_2^2h_3^3,\qquad h_1,h_2\in\Hq,\quad(h_1,h_2)=1,
\]
where $h_3$ need not be coprime to $h_1h_2$.  Put $m=g/2+1$ and
\[
 C_q=\frac{\zeta_q(3/2)}{\zeta_q(3)}
       \mathcal A_{\mathrm{nK}}(q^{-2},q^{-3/2}).
\]
Define
\begin{equation}\label{tw:C}
 C_q(h)=\frac{C_q}{|h_1|\,|h_2|^{1/2}}
 \prod_{\substack{P\mid h\\\deg P\ {\rm even}}}
 \left(1+\frac2{|P|}-\frac2{|P|^{5/2}}\right)^{-1}.
\end{equation}
Here and below products over $P$ are over monic primes in $\Fq[T]$.
For such a prime put $R=|P|$, $d=\deg P$, and
\[
 D_{P,j}=
 \begin{cases}
 1+R^{-5/6}\omega^{-jd},&d\text{ odd},\\
 1+2R^{-1}+R^{-5/6}\omega^{-jd}-2R^{-5/2},&d\text{ even}.
 \end{cases}
\]
Define
\begin{equation}\label{tw:local-secondary}
 c_j(h)=c_j\prod_{P^e\parallel h}\left[
 \frac1{D_{P,j}}
 \begin{cases}
 1+R^{-5/6}\omega^{-jd},&e\equiv0\pmod3,\\
 R^{-1/2}(1+R^{-5/6}\omega^{-jd}),&e\equiv1\pmod3,\\
 R^{-1/3}\omega^{-jd}+R^{-1},&e\equiv2\pmod3
 \end{cases}\right],
\end{equation}
where $c_j$ is defined in \eqref{eq:cj}.

\subsection{The dual series with a twist}
Define
\[
 A_{3,h}(s,w)=
 \sum_{\substack{F\in\Hqq\\
 F\text{ has no nonconstant divisor in }\Fq[T]}}
 \frac{\chi_F(h)L_q(w,\chi_F)}{|F|_2^s}.
\]
The character is zero when $(F,h)\ne1$. For $k=s+w$,
the functional equation and cubic twisting give
\begin{equation}\label{tw:dual}
 A_{3,h}(s,w)
 =\frac{q^{2w-1}(1-q^{-w})}{1-q^{w-1}}
 \sum_{\substack{N,D\in\Aq\\(N,D)=1,\ (D,h)=1}}
 |N|^{w-1}\mu(D)q^{\deg D}|D|_2^{-k}H(NDh^2,k).
\end{equation}
Indeed $\chi_F(h)\overline{\chi_F(ND)}=
\overline{\chi_F(NDh^2)}$ when $(F,h)=1$.
Since $D,h\in\Fq[T]$, the value $\chi_D(h)$ equals $1$ if $(D,h)=1$ and is
zero otherwise.  There is no condition $(N,h)=1$.
This identity initially holds on the same overlap of absolute convergence as \eqref{eq:dual}.

In the following block estimates, the index $m$ denotes $\deg D$, as in Section~3.
Define $B_{nm,h}$ by replacing $H(ND,k)$ with $H(NDh^2,k)$ in
\eqref{eq:Bnm} and imposing $(D,h)=1$, and put
$\Delta B_{nm,h}=B_{nm,h}-qB_{n-1,m,h}$.
The identity \eqref{eq:diff} holds with these
replacements and with $A_3$ replaced by $A_{3,h}$.

On opening the Gauss series, only $F$ coprime to $h$ remain.
For $F=EG$ as in Section~3, the extra factor is
$\chi_F(h)$, of modulus one.  The product in \eqref{eq:SFproduct} becomes
\[
 R_F(x)\prod_{P\nmid Eh}
 \bigl(1-\chi(P)y^{\deg P}+\chi(P)^2(xy)^{\deg P}\bigr),
 \qquad \chi=\overline{\chi_G}|_{\Fq[T]}.
\]
Its factorization as in \eqref{eq:SFfactor} has the additional
finite factors $(1-\chi(P)y^{\deg P})^{-1}$ for $P\mid h$, while
those primes are omitted from the convergent correction product.
On the coefficient torus \eqref{eq:torus}, these finite factors are
$\ll_{q,\delta,\eps}|h|^\eps$.  The same argument used for the
prime factors of $E$ proves this bound.  Therefore
\begin{equation}\label{tw:right-endpoint}
 |\Delta B_{nm,h}(3/2+\kappa+i\theta)|
 \ll_{q,\kappa,\eps}
 |h|^\eps q^{n/2+(3/2+\eps)m}.
\end{equation}

At $a=2/3+\kappa$, Lemma~\ref{pre:gauss-series} instead gives
\begin{equation}\label{tw:left-endpoint}
 |\Delta B_{nm,h}(a+i\theta)|
 \ll_{q,\kappa,\eps}
 |h|^{3-2a+\eps}
 q^{(5/2-a+\eps)n+(7/2-a+\eps)m}.
\end{equation}
Here $|NDh^2|_2=q^{2(n+m)}|h|^4$.
The interpolation on displaced annuli from Section~3 consequently gives
\begin{equation}\label{tw:block}
 |\Delta B_{nm,h}(k)|
 \ll_{q,I,\eps}
 |h|^{3-2\Re k+\eps}
 q^{(29/10-8\Re k/5+\eps)n+(39/10-8\Re k/5+\eps)m},
\end{equation}
uniformly on compact $I\subset(2/3,3/2)$. Thus $R_{3,D}$ remains
available, with a uniform bound
$\ll|h|^{3-2\Re(s+w)+\eps}$ for the pole-free product \eqref{eq:commonproduct} with $A_3$ replaced by $A_{3,h}$ on
each compact real subregion and all imaginary parts.

\subsection{Uniform control through the convex hull}
Write
\[
 F_h(s,w)=(1-q^{2-2s})(1-q^{w-1})
          (1-q^{8-6(s+w)})A_{3,h}(s,w).
\]
On the initial region $R_2$ the original series and
$|\chi_F(h)|\le1$ give bounds independent of $h$, apart from the
term with conductor $1$, whose pole has already been removed.
On every compact real subregion of $R_1$ one has, uniformly in the
imaginary parts,
\begin{equation}\label{tw:R1-bound}
 |F_h(s,w)|\ll_{q,\eps}|h|^\eps.
\end{equation}
The implicit constant also depends on the compact subregion.
To prove this, interchange sums as in \cite[Section~3.2]{HF},
replacing $N$ by $hN$ inside the character. For a prime of even degree
$P=\pi\pi^\sigma$, put $a=\chi_\pi(hN)$,
$b=\chi_{\pi^\sigma}(hN)$ and $z=u^{\deg P/2}$.
Extracting $L_{q^2}(s,\chi_{hN})$ from the family sum leaves
\[
 (1+(a+b)z)(1-az)(1-bz)=1+O(|u|^{\deg P});
\]
at a prime of odd degree the correction is
$1-\chi_P(hN)u^{\deg P}$. The correction product is therefore
uniformly bounded on compact subsets of $|u|<q^{-1}$.
The inducing conductor divides $hN$, so Lindel\"of and the finite
imprimitive factors give $O(|hN|^\eta)$ for every $\eta>0$.
If $hN$ is a cube, $1-q^{2-2s}$ removes the principal pole and
the same bound holds. Summing $|N|^{-\Re w+\eta}$ with
$\eta<\Re w-1$ proves \eqref{tw:R1-bound}.

The overlap argument of Section~3 now continues $F_h$ to $T(R_4)$
for each $h$.  To propagate its bounds, fix a small $\delta>0$ and
put
\[
 A_\delta=(1/2+\delta,1+\delta),\qquad
 G_\delta=(5/4+\delta,-1/6+\delta).
\]
The first point lies in $R_1$ and the second in $R_{3,D}$.
Their endpoint bounds are
\[
 |F_h(A_\delta+i\boldsymbol\theta)|\ll |h|^\eta,
 \qquad
 |F_h(G_\delta+i\boldsymbol\theta)|
       \ll |h|^{5/6-4\delta+\eta},
\]
uniformly for $\boldsymbol\theta\in\mathbb R^2$.
Apply Hadamard's three lines theorem to
\[
 z\longmapsto
 F_h\bigl(A_\delta+z(G_\delta-A_\delta)
                         +i\boldsymbol\theta\bigr),
 \qquad 0\le\Re z\le1.
\]
For fixed $h$, Remark~\ref{pre:tube-application} bounds this strip
function, since its real parts lie on the compact segment joining
$A_\delta$ and $G_\delta$. The three lines theorem therefore
propagates the uniform endpoint bounds.

For
\[
 \lambda=\frac37+\frac{6\delta}7,
 \qquad
 (1-\lambda)A_\delta+\lambda G_\delta
 =\left(\frac{23}{28}+\frac{23\delta}{14},\frac12\right),
\]
the exponent of $|h|$ is
\[
 \lambda(5/6-4\delta)+\eta\le\frac5{14}+\eta.
\]
Returning to $m=g/2+1$, choose $\delta$ and $\eta$ sufficiently small in terms of the
final $\eps$. This gives a uniform bound for $F_h$ on the central
coefficient circle with radius
$q^{-2(23/28+23\delta/14)}$.
After division by the clearing factors, which are bounded away from
zero on that circle, its integral in the coefficient formula is
\begin{equation}\label{tw:error}
 O_{q,\eps}\!\left(q^{(23/14+\eps)m}|h|^{5/14+\eps}\right).
\end{equation}

\subsection{The first and secondary residues}
The first residue in the interchanged expression comes from $hN$
being a cube.  Put
\[
 c_q=\prod_{\deg P\ {\rm odd}}(1-|P|^{-2})
     \prod_{\deg P\ {\rm even}}
              (1-|P|^{-1})^2(1+2|P|^{-1}).
\]
This is the leading constant for the number of primitive characters:
the generating product for the family is
$\prod_{\deg P\ {\rm even}}(1+2u^{\deg P/2})$.
Excluding a prime of even degree $P$ from the conductor multiplies its
normalized residue at $u=q^{-2}$ by $(1+2/|P|)^{-1}$.
Consequently the first residue constant is equivalently
\begin{equation}\label{tw:C-diagonal}
 C_q(h)=c_q
 \sum_{\substack{N\in\Aq\\hN\ {\rm a\ cube}}}
 \frac1{|N|^{1/2}}
 \prod_{\substack{P\mid hN\\\deg P\ {\rm even}}}
 (1+2/|P|)^{-1}.
\end{equation}
The condition forces $N=h_1^2h_2K^3$.
Summing over $K$ prime by prime and comparing with $h=1$ gives
exactly \eqref{tw:C}.  This also proves $C_q(h)>0$ and
\[
 \frac{|h|^{-\eps}}{|h_1||h_2|^{1/2}}
 \ll_{q,\eps} C_q(h)
 \le \frac{C_q}{|h_1||h_2|^{1/2}}.
\]

For the secondary residues, apply Lemma~\ref{pre:rational-residue}
to $NDh^2$.  The termwise evaluation is justified first for
$\Re w<0$ by \eqref{tw:block}, exactly as in
Section~\ref{subsec:full-residues}; the resulting residue series is
absolutely convergent for $\Re w<1/3$ for each fixed $h$.
The residue identity then continues along the same divisor to
$w=1/2$.  It remains to describe what changes at a prime dividing $h$.

Fix $P^e\parallel h$. This prime cannot divide $D$; if its
exponent in $N$ is $n$, its exponent in $NDh^2$ is $n+2e$.
Lemma~\ref{pre:rational-residue} and \eqref{pre:beta} give the
following quotient of the twisted and untwisted local residues:
\[
 \frac{1-R^{-3/2}}{D_{P,j}}\left(
 \sum_{\substack{n\ge0\\n+2e\equiv0\ (3)}}R^{-n/2}
 +R^{-1/3}\omega^{-jd}
 \sum_{\substack{n\ge0\\n+2e\equiv1\ (3)}}R^{-n/2}\right).
\]
Evaluating the geometric sums gives \eqref{tw:local-secondary}.
Thus $c_j(h)$ is the normalized residue of $A_{3,h}$ at
$u=q^{-5/3}\omega^j$.

All factors in \eqref{tw:local-secondary} are nonzero. For the
third numerator this follows from $R^{-1/3}>R^{-1}$.
The other numerator and denominator assertions follow as in the
untwisted calculation.  After removing $R^{-1/2}$ in residue class
$1$ and $R^{-1/3}\omega^{-jd}$ in residue class $2$, each local
factor is $1+O(R^{-2/3})$, uniformly in $j$ and $e$.
A finite product over the distinct primes dividing $h$ is then
$\ll_{q,\eps}|h|^\eps$, by separating finitely many primes of small degree from the remaining ones.  This proves
\eqref{tw:secondary-bound}.  The same argument with reciprocals
also supplies a lower bound with $|h|^{-\eps}$.
For $j=0$ all finite factors are positive, hence $c_0(h)<0$;
conjugation gives $c_2(h)=\overline{c_1(h)}$.
The contour argument of Section~3, with \eqref{tw:error}, now proves
Theorem~\ref{thm:twisted}.

\subsection{Summation against short Dirichlet polynomials}
Suppose that complex coefficients $a(h)$ satisfy
$|a(h)|\ll_\eta |h|^\eta$ for every $\eta>0$, with constants
uniform over the polynomials under consideration.  They may depend
on $m$, provided these bounds remain uniform.  For every $L\ge0$,
Theorem~\ref{thm:twisted} implies
\begin{equation}\label{tw:summed-error}
 \begin{split}
 &\sum_{\deg h\le L}\frac{a(h)}{|h|^{1/2}}
   \sum_{\substack{\chi\ {\rm primitive\ cubic}\\
                   \operatorname{genus}(\chi)=g}}
     \chi(h)L_q(1/2,\chi)\\
 &\quad=q^{2m}\sum_{\deg h\le L}
                    \frac{a(h)C_q(h)}{|h|^{1/2}}
 +O_{q,\eps,a}\!\left(
 q^{5m/3+\eps L}
 +q^{(23/14+\eps)m+(6/7+\eps)L}\right).
 \end{split}
\end{equation}
To bound the secondary sum, write
$\deg h_1=a$, $\deg h_2=b$, and $\deg h_3=c$.
There are at most $q^{a+b+c}$ such triples.  Combining
$|h|^{-1/2}$ with \eqref{tw:secondary-bound} leaves the factor
\[
 q^{a+b+c}q^{-(a+2b+3c)/2}q^{-a/2-b/3}
 =q^{-b/3-c/2}.
\]
Summation over $a+2b+3c\le L$ is $O_q(L+1)$, which is absorbed
into $q^{\eps L}$ after rescaling $\eps$.
For the remainder, summing by $d=\deg h$ gives
\[
 \sum_{d\le L}q^dq^{-d/2}q^{(5/14+\eps)d}
 \ll_{q,\eps}q^{(6/7+\eps)L}.
\]
Thus, for every fixed $\delta>0$, the error in
\eqref{tw:summed-error} is $O(q^{(2-\eta)m})$ for some $\eta>0$
whenever $L\le(5/12-\delta)m$, with $\eps$ chosen sufficiently
small.

\end{document}